\documentclass[12pt, leqno]{amsart}
\usepackage{graphicx}
\usepackage{amsfonts,delarray,amssymb,amsmath,amsthm,a4,a4wide}
\usepackage{latexsym}
\usepackage{epsfig}
\usepackage{color}
\usepackage[margin=1in]{geometry} 

\newtheorem{thm}{Theorem}[section]
\newtheorem{cor}[thm]{Corollary}
\newtheorem{conj}[thm]{Conjecture}
\newtheorem{lem}[thm]{Lemma}
\newtheorem{exam}[thm]{Example}

\newtheorem{rem}[thm]{Remark}

\theoremstyle{definition}
\newtheorem{defn}[thm]{Definition}

\numberwithin{equation}{section}

\newcommand{\diag}{\text{diag}}

\newcommand{\R}{\mathbb R}

\newcommand{\e}{\varepsilon}

\newcommand{\p}{\partial}

\newcommand{\comment}[1]{}

{\begin{list}{}%
         {\setlength{\leftmargin}{#1}}%
         \item[]%
}
{\end{list}}
\makeatletter
\@namedef{subjclassname@2020}{%
  \textup{2020} Mathematics Subject Classification}
\makeatother
\begin{document} 
 \dedicatory{Dedicated to Professor Duong Minh Duc on the occasion of his 70th birthday}

\title[Hadamard-type inequalities]{Hadamard-type inequalities for $k$-positive matrices}
\author{Nam Q. Le}
\address{Department of Mathematics, Indiana University, 831 E 3rd St,
Bloomington, IN 47405, USA}
\email{nqle@indiana.edu}
\thanks{The author was supported in part by the National Science Foundation under grant DMS-2054686.}



\maketitle

\begin{abstract}
We establish Hadamard-type inequalities for a class of symmetric matrices called $k$-positive matrices for which the $m$-th elementary symmetric functions of their eigenvalues are positive for all $m\leq k$.   These matrices arise naturally in the study of $k$-Hessian equations  
in Partial Differential Equations.
For each  $k$-positive matrix, we show that the sum of its principal minors of size $k$ is not larger than the $k$-th elementary symmetric function of their diagonal entries.
The case $k=n$ corresponds to the classical Hadamard inequality for positive definite matrices. Some consequences are also obtained. 
\end{abstract}

\section{Introduction}

Let $n\geq 2$ and $1\leq k\leq n$. We denote the $k$-th symmetric function of $n$ variables $\lambda=(\lambda_1,\cdots,\lambda_n)\in \R^n$ by
$$S_k(\lambda):=\sum_{1\leq i_1<\cdots<i_k\leq n}\lambda_{i_1}\cdots \lambda_{i_k}.$$
It is convenient to set $$S_0(\lambda)=1.$$
Let $\Gamma_k(n)$ be an open symmetric convex cone in $\R^n$, with vertex at the origin, given by
$$\Gamma_k(n)=\{\lambda=(\lambda_1, \cdots,\lambda_n)\in \R^n\mid S_j(\lambda)>0\quad\forall j=1, \cdots, k\}.$$
The convexity of $\Gamma_k(n)$ is a consequence of G\r{a}rding's theory of hyperbolic polynomials; see Example \ref{Ex}.

Let $M_n(\R)$ be the set of $n\times n$ matrices with real entries. 
If  $A=(a_{ij})_{1\leq i, j\leq n}\in M_n(\R)$ is an $n\times n$ symmetric matrix, we use $\lambda(A)=(\lambda_1, \cdots,\lambda_n)$ to denote  its eigenvalues. For $A\in M_n(\R)$, let $\diag(A)$ be its diagonal matrix:
$$\diag(A)=\text{diag}(a_{11}, \cdots, a_{nn}).$$
\vglue 0.2cm
{\bf Notation.} We use the following notation:
$$S_k(A) =S_k(\lambda(A));$$
$$[n]=\{1, \cdots, n\};\quad J^c=[n]\setminus J\text{ for } J\subset [n].$$
For $J\subset [n]$, we denote by
$A[J]$ the principal submatrix of $A$ of size $|J|$ obtained by deleting the $i$th row and column of $A$, for each $i\not \in J$. 
\vglue 0.2cm
Let $E_k(A)$ be the sum of the principal minors of size $k$ of $A\in M_n(\R)$. Then, by \cite[Theorem 1.2.16]{HJ}, we have
\begin{equation}
\label{SEk}
S_k(A) = E_k(A).
\end{equation}
If $A=(a_{ij})_{1\leq i, j\leq n}\in M_n(\R)$ is positive definite, or equivalently, $\lambda(A)\in \Gamma_n(n)$, then Hadamard's determinant inequality (see, for example, \cite[Theorem 7.8.1]{HJ}) gives
\begin{equation}
\label{SnDA}
S_n(\diag(A))= a_{11}\cdots a_{nn}\geq \det A= S_n(A).
\end{equation}
When $A\in M_n(\R)$ is positive definite, and $1\leq k\leq n$ is fixed, each principal submatrix  of size $k$ of $A$ is also positive definite; thus, we can apply the Hadamard inequality to each of these principal submatrices of $A$ and use (\ref{SEk}) to conclude that 
\begin{equation}
\label{SkDA}S_k(\diag(A)) \geq S_k(A).
\end{equation}
In analogy with the classical Hadamard inequality (\ref{SnDA}), we call (\ref{SkDA}) a Hadamard-type inequality.

In this note, we show that (\ref{SkDA}) holds for a larger class of symmetric matrices, called $k$-positive.
\begin{defn}[$k$-positive matrices] Let $1\leq k\leq n$. A  symmetric $n\times n$ matrix $A\in M_n(\R)$ is call $k$-positive if $\lambda(A)\in\Gamma_k(n)$.
\end{defn}
As will be seen in Example \ref{Pkex}, the set of $k$-positive matrices is a convex cone. This is again a consequence of G\r{a}rding's theory of hyperbolic polynomials.
\vglue 0.2cm

Note that the class of $n$-positive matrices is equal to the class of positive definite matrices. The class of $k$-positive matrices arises naturally in the study of $k$-Hessian equations  
$$S_k(D^2 u)=f$$
in Partial Differential Equations where $D^2 u$ denotes the Hessian matrix of $u$; see \cite{IF} for a survey.
\vglue 0.2cm
Due to the following remark, we will focus on the case $k\geq 3$.
\begin{rem} 
\label{rem12}
Let $A=(a_{ij})_{1\leq i, j\leq n}\in M_n(\R)$ be symmetric.
\begin{enumerate}
\item[(i)] If $k=1$, then
$$S_1(A)= \sum_{i=1}^n \lambda_i(A)=\sum_{i=1}^n a_{ii}= S_1 (\diag(A)).$$
\item[(ii)] If $k=2$, then
\begin{eqnarray*}S_2(A) =E_2(A)&=&\sum_{1\leq i<j\leq n} a_{ii} a_{jj}-\sum_{1\leq i<j\leq n} a^2_{ij}\\
&=&S_2(\diag(A)) -\sum_{1\leq i<j\leq n} a^2_{ij}\leq S_2(\diag(A)).
\end{eqnarray*}
Equality holds if and only if $A$ is diagonal.
\end{enumerate}
\end{rem}
Our main result on Hadamard-type inequalities for $k$-positive matrices states as follows.
\begin{thm}[Hadamard-type inequalities for $k$-positive matrices]
\label{Hnk} 
Let $n\geq k\geq  3$.  Let $A\in M_n(\R)$ be $k$-positive. Then $S_k(\diag(A)) \geq S_k(A)$. Moreover, equality holds if and only if $A$ is diagonal.
\end{thm}
A simple corollary of Theorem \ref{Hnk} and Remark \ref{rem12} is the following.
\begin{cor}
\label{cor1}
Let $n\geq k\geq  1$.  Let $A=(a_{ij})_{1\leq i, j\leq n}\in M_n(\R)$ be $k$-positive. Then $\diag(A)$ is $k$-positive. In other words, $(a_{11}, \cdots, a_{nn})\in \Gamma_k(n)$. Moreover, $S_k(\diag(A))\geq S_k(A)$.
\end{cor}
\noindent
For $p\in [n]$ and $\lambda =(\lambda_1,\cdots,\lambda_n)\equiv (\lambda_i)_{1\leq i\leq n}\in\R^n$, let us denote the following point in $\R^{n\choose p}$:
 $$\displaystyle \lambda_{[p]}=\left( \lambda_{i_1}+\cdots+\lambda_{i_p}\right)_{1\leq i_1<\cdots<i_p\leq n}.$$
Note that $\lambda_{[1]}=\lambda$. We now state an interesting consequence of Corollary \ref{cor1}.
 \begin{thm}
 \label{pcor}
 Let $A=(a_{ij})_{1\leq i, j\leq n}\in M_n(\R)$ be symmetric. Let $p\in [n]$ and $1\leq k\leq {n\choose p}$. If $\lambda(A)_{[p]} \in \Gamma_k ({n\choose p})$ then $(a_{11},\cdots, a_{nn})_{[p]}\in \Gamma_k ({n\choose p})$ and 
 $S_k((a_{11},\cdots, a_{nn})_{[p]})\geq S_k(\lambda(A)_{[p]}).$
 \end{thm}
We deduce from Theorem \ref{Hnk} the following result.
\begin{cor}
\label{cor2}
Let $n\geq k\geq  2$. Let $A=(a_{ij})_{1\leq i, j\leq n}$, and $B=(b_{ij})_{1\leq i, j\leq n}\in M_n(\R)$ be two $k$-positive matrices. Then
$$\sum_{i=1}^n b_{ii} S_{k-1}(A[\{i\}^c])\geq k [S_k(A)]^{\frac{k-1}{k}}  [S_k(B)]^{\frac{1}{k}}.$$
\end{cor}
The rest of this note is organized as follows. In Section \ref{Hpf}, we prove Theorem \ref{Hnk}. In Section \ref{pcor_sec}, we prove Theorem \ref{pcor}. The proof of Corollary \ref{cor2} will be given in Section \ref{Gsec}. The final Section \ref{HypSec} relates the main results and concepts of this note with hyperbolic polynomials.
\section{Proof of Theorem \ref{Hnk}}
\label{Hpf}
In this section, the entries of $A\in M_n(\R)$ will be denoted by $a_{ij}$ so $A=(a_{ij})_{1\leq i, j\leq n}$.

We start with the following useful expansion.
\begin{lem}
\label{expandlem}
Let $A\in M_n(\R)$ be symmetric.
If $A[\{n\}^c]$ is diagonal, then for $k\geq 2$, we have
\begin{eqnarray*}S_k(A)&=& S_k(\diag(A)) -\sum_{i<n} a^2_{in} \left(\sum_{i_1<\cdots<i_{k-2}\in\{i, n\}^c } a_{i_1 i_1}\cdots a_{i_{k-2} i_{k-2}}\right)\\&\equiv& S_k(\diag(A)) -\sum_{i<n} a^2_{in} S_{k-2}(\diag(A[\{i, n\}^c])).
\end{eqnarray*}
\end{lem}
\begin{proof}
Recall that $S_k(A)$ is the sum of the principle minors of size $k$ of $A$. Using the definition of determinant of $k\times k$ matrices together with the fact that $A[\{n\}^c]$ is diagonal, we find
\begin{eqnarray*}
S_k(A)&=& S_k(\diag(A)) \\&&+ \sum_{i<n} a^2_{in}\left(\sum_{i_1<\cdots<i_{k-2}\in\{i, n\}^c }  \text{sign}  \begin{pmatrix}
i & n & i_1& \cdots & i_{k-2} \\
n & i & i_1&\cdots & i_{k-2} 
\end{pmatrix} a_{i_1 i_1}\cdots a_{i_{k-2} i_{k-2}}\right)\\
&=&S_k(\diag(A)) -\sum_{i<n} a^2_{in}\left(\sum_{i_1<\cdots<i_{k-2}\in\{i, n\}^c } a_{i_1 i_1}\cdots a_{i_{k-2} i_{k-2}}\right).
\end{eqnarray*}
Here 
$$ \text{sign}  \begin{pmatrix}
i & n & i_1& \cdots & i_{k-2} \\
n & i & i_1&\cdots & i_{k-2} 
\end{pmatrix}=-1$$
is the sign of the permutation of $k$ numbers $i, n, i_1,\cdots, i_{k-2}$.
\end{proof}
  
Our key lemma in the proof of Theorem \ref{Hnk} is the following.
 \begin{lem}
   \label{S3nn}
   Let $n>k\geq 2$ and let $A\in M_n(\R)$ be symmetric.  Let $j\in [n]$. Assume that $A[\{j\}^c])$ is $(k-1)$-positive\footnote{In this revision of the published version of this article in {\it Linear Algebra Appl.} {\bf 635} (2022), 159-170, the assumption in Lemma \ref{S3nn} was modified to make it invariant under conjugation with orthogonal matrices. All arguments and results remain unchanged.
}. 
 Then
$$ S_k(A)\leq S_k(A[\{j\}^c]) + a_{jj} S_{k-1} (A[\{j\}^c]).$$
Moreover, the equality holds if and only if $a_{ij}=0$ for all $i\neq j$.
 \end{lem}
 \begin{proof}
 We can assume that $j=n$. Then, for all $i<n$, we have, by Theorem \ref{iner_thm} below, $$S_{k-2}(A[\{i, j\}^c])>0.$$

  {\bf Case 1.} Consider the case $A[\{n\}^c]:=(a_{ij})_{1\leq i, j\leq n-1}$ is diagonal. Then, from Lemma \ref{expandlem}, we have
  \begin{eqnarray*}
  S_k(A[\{n\}^c]) + a_{nn} S_{k-1}(A[\{n\}^c])-S_k(A)&=& S_k(\diag(A))-S_k(A)\\&=&  \sum_{i<n} S_{k-2} (\diag(A[\{i, n\}^c])) a^2_{in}\geq 0.
    \end{eqnarray*}
    Moreover, the equality holds if and only if $a_{in}=0$ for all $i< n$.

 {\bf Case 2.} General case.  We can find an orthogonal matrix $U\in O(n-1)$ such that
  $U^t A[\{n\}^c] U$
  is diagonal. Let $$W= U\bigoplus 1:=
  \left( \begin{array}{cc}
U & 0   \\
0   & 1 \end{array} \right) \in O(n)$$ and $B= (a_{in})_{1\leq i\leq n-1}$. 
  Then $$W^t A W=\left( \begin{array}{cc}
U^t A[\{n\}^c] U & U^t B   \\
 B^t U  & a_{nn} \end{array} \right)$$ has the form considered in {\bf Case 1}. Note that
$S_m(W^t A W)= S_m(A)>0$ for  $1\leq m\leq k.$
Therefore, from {\bf Case 1}, we have
  \begin{eqnarray*}
  S_k(A)= S_k(W^t A W) &\leq& S_k (U^t A[\{n\}^c] U) + a_{nn} S_{k-1}(U^t A[\{n\}^c] U)\\&=& S_k (A[\{n\}^c]) + a_{nn} S_{k-1}(A[\{n\}^c]).
  \end{eqnarray*}
  The equality occurs if and only if $U^t B=0$, or equivalently, $a_{in}=0$ for all $i<n$.
   \end{proof}
   The key assumption in Lemma \ref{S3nn} can be deduced, in many cases, 
from the following result which is a consequence of Sylvestre's criterion established in \cite[Theorem 2.1]{IF}.
\begin{thm}[Theorem 2.1 in \cite{IF}]
\label{iner_thm}
Let $A\in M_n(\R)$ be $k$-positive
  where $k\geq 2$. Then for all $i\in [n]$, we have that $A[\{i\}^c]$ is $(k-1)$-positive.
\end{thm}
For reader's convenience, we provide a different proof of Theorem \ref{iner_thm} using G\r{a}rding's inequality in Section \ref{Gsec}. 

We begin the proof of Theorem \ref{Hnk} with the case $k=3$.
\begin{lem}
\label{Hn3} Let $n\geq  4$.  Let $A\in M_n(\R)$ be $3$-positive. Then $S_3(\diag(A)) \geq S_3(A)$. Moreover, equality holds if and only if $A$ is diagonal.
\end{lem}
 \begin{proof}[Proof] 
Fix $j\in[n]$.
Since $A$ is $3$-positive, we can apply  Theorem \ref{iner_thm} twice to find that if 
$i\neq j$, then $A[\{i, j\}^c]$ is $1$-positive.
Thus
\begin{equation}
\label{1kk}
\sum_{k\in\{i, j\}^c} a_{kk}=S_1 (\diag(A[\{i, j\}^c]))>0.
\end{equation}
From $A[\{j\}^c]$ being $2$-positive and Lemma \ref{S3nn}, we have
$$S_3(A) \leq S_3(A[\{j\}^c]) + a_{jj} S_2 (A[\{j\}^c]).$$
Adding these inequalities, and noting that
$$(n-3)S_3(A)=\sum_{j=1}^n S_3(A[\{j\}^c]), $$
we find
\begin{eqnarray*}3S_3(A) \leq \sum_{i=1}^n a_{ii}  S_2(A[\{i\}^c])
&=& 3S_3 (\text{diag}(A)) - \sum_{i=1}^n \left(  a_{ii}\sum_{i\neq j\neq k\neq i}a^2_{jk}\right)\\
&=& 3S_3(\diag(A)) - \sum_{i<j} \left(a_{ij}^2 \sum_{k\in\{i, j\}^c} a_{kk}\right)\\&\leq& 3S_3(\diag(A))
\end{eqnarray*}
where we used (\ref{1kk}) in the last inequality.
Clearly, equality occurs if and only if $a_{ij}=0$ for all $i\neq j$ or if $A$ is diagonal.
 \end{proof}
 We are now ready to prove Theorem \ref{Hnk}.
 \begin{proof}[Proof of Theorem \ref{Hnk}] 
 As remarked in the introduction, we have $S_2(\diag(A))\geq S_2(A)$ for any symmetric matrix $A\in M_n(\R)$ with equality holding if and only if $A$ is diagonal. We only consider the case $k<n$ since the case $k=n$ is the classical Hadamard inequality.
 
 The proof of the theorem is by induction on $k\geq 3$, the base case being Lemma \ref{Hn3}. Suppose that the theorem is true up to $k\geq 3$. We prove it for $k+1<n$. 
 
 Assume $A\in M_n(\R)$ is $(k+1)$-positive.
 Then, by Theorem \ref{iner_thm}, 
 $A[\{i\}^c]$ is $m$-positive for $1\leq m\leq k$.
 For each $j\in [n]$, let $A^{j,0}$ be the matrix obtained from $A$ by replacing all entries in the $j$-th row and column by $0$, except $a_{jj}$ being kept unchanged. 
 
 {\it Step 1.} We show that $A^{n,0}$ is $(k+1)$-positive. Indeed, from $A[\{i\}^c]$ being $m$-positive
 for all $i\in [n]$, we find that the hypothesis of Lemma \ref{S3nn} is satisfied where $k$ there being replaced by $(m+1)$ here.
  We can then apply Lemma \ref{S3nn} to find that, for $1\leq m\leq k$, we have
 $$S_{m+1} (A^{n, 0})= S_{m+1}(A[\{n\}^c])) + a_{nn} S_m(A[\{n\}^c]))\geq S_{m+1}(A)>0$$
 with equality if and only if $a_{in}=0$ for all $i<n$. This combined with $S_{1} (A^{n, 0})= S_{1}(A)>0$ shows that $A^{n, 0}$ is $(k+1)$-positive . 
 
 {\it Step 2.} Next, for each $i\in[n-1]$, we replace the non-diagonal term in the $i$-th row and column of $A^{n, 0}$ by $0$, we obtain a new $(k+1)$-positive matrix with no less $S_{k+1}$ value. 
 Repeating this process, we obtain the conclusion of the theorem for $k+1$  with equality if and only if $A$ is diagonal.
 \end{proof}
\section{Proof of Theorem \ref{pcor}}
\label{pcor_sec}
In this section, we prove Theorem \ref{pcor}.
 The proof uses ideas from Harvey-Lawson \cite{HL2} to interpret $\lambda(A)_{[p]}$ as eigenvalues of a suitable matrix associated with $A$. We recall this formalism.
 
 Let $Sym^2(\R^n)$ be the space of symmetric endomorphisms of $\R^n$. Fix an orthonormal basis $(e_1,\cdots, e_n)$ of $\R^n$. For $p\in [n]$, let $\Lambda^p \R^n$ be the space of $p$-vectors $v_1\wedge\cdots\wedge v_p$ where $v_i\in\R^n$ for $1\leq i\leq p$. The inner product on $\R^n$ induces an inner product on $\Lambda^p\R^n$. Then, an induced orthonormal basis for $\Lambda^p\R^n$ is $\{e_{i_1}\wedge\cdots\wedge e_{i_p}\}$ where $(i_1,\cdots, i_p)$ runs over all increasing $p$-tuples which are ordered lexicographically.  
 
 For each symmetric matrix $A\in M_n(\R)$, we can view its as a member of $Sym^2 (\R^n)$. We define the linear derivation $\mathcal{D}_A$ of $A$ on $\Lambda^p\R^n$ by
 assigning each $p$-vector $v_1\wedge\cdots\wedge v_p\in \Lambda^p\R^n$ another $p$-vector
 $$\mathcal{D}_A(v_1\wedge\cdots\wedge v_p)= Av_1\wedge\cdots\wedge v_p + v_1\wedge Av_2\wedge\cdots\wedge v_p +\cdots + v_1\wedge\cdots\wedge Av_p\in \Lambda^p\R^n.$$
 Clearly, $\mathcal{D}_A\in Sym^2(\Lambda^p\R^n)$, and $\mathcal{D}_A$ has a matrix representation with respect to the induced basis $\{e_{i_1}\wedge\cdots\wedge e_{i_p}\}$ with matrix entries being linear combinations of the entries of $A$. Moreover,
 \begin{equation}
 \label{diagDA} \diag (\mathcal{D}_A)= \mathcal{D}_{\diag(A)}=\diag \left(a_{i_1 i_1}+\cdots +a_{i_p i_p}\right)_{1\leq i_1<\cdots<i_p\leq n}.
 \end{equation}
 
 In \cite[Lemma 2.5]{HL2}, Harvey and Lawson showed that if $A$ has eigenvalues $\lambda(A)=(\lambda_1,\cdots,\lambda_n)$ with corresponding  eigenvectors $(v_1,\cdots v_n)$, then $\mathcal{D}_A$ has eigenvalues $$\{\lambda_{i_1}+\cdots +\lambda_{i_p}: 1\leq i_1<\cdots<i_p\leq n\},$$ with corresponding eigenvectors $$\{ v_{i_1}\wedge\cdots\wedge v_{i_p}: 1\leq i_1<\cdots<i_p\leq n\}.$$
 Thus, in our notation, 
 
 \begin{equation}
 \label{notaDA}
 \lambda(\mathcal{D}_A)= \lambda(A)_{[p]}\quad \text{and } S_k(\lambda(A)_{[p]})= S_k(\mathcal{D}_A).
 \end{equation}
  \begin{proof}[Proof of Theorem \ref{pcor}] We use the above setup and notation. 
 If $\lambda(A)_{[p]} \in \Gamma_k ({n\choose p})$, then 
 $\lambda(\mathcal{D}_A) \in \Gamma_k \left({n\choose p}\right).$
 By Corollary \ref{cor1}, we then have
 $\diag(\mathcal{D}_A)\in  \Gamma_k \left({n\choose p}\right)$
 and
 $$S_k (\diag(\mathcal{D}_A))\geq S_k(\mathcal{D}_A).$$
 In view of (\ref{diagDA}) and (\ref{notaDA}), we obtain the conclusion of the theorem.
 \end{proof}
\section{Proofs of Theorem \ref{iner_thm} and Corollary \ref{cor2} via G\r{a}rding's inequality}
\label{Gsec}
In the proofs of Theorem \ref{iner_thm} and Corollary \ref{cor2}, 
we will use the following form of G\r{a}rding's inequality \cite{G}.
\begin{lem}[G\r{a}rding's inequality]
\label{Glem}
Suppose that $A=(a_{ij})_{1\leq i, j\leq n}$, and $D=(d_{ij})_{1\leq i, j\leq n}\in M_n(\R)$ are two $k$-positive matrices. Then
\begin{equation}
\label{Gij}
\sum_{i, j=1}^n d_{ij} S^{ij}_k (A)\geq k [S_k(A)]^{\frac{k-1}{k}}  [S_k(D)]^{\frac{1}{k}}
\text
{ where }
S^{ij}_k(A)=\frac{\p}{\p a_{ij}} S_k(A).
\end{equation}
\end{lem}
Lemma \ref{Glem} follows from the polarization inequality in \cite[Theorem 5]{G} for the polynomial $S_k(A)$; see also \cite[inequality (3.2)]{L} for a related version when $A$ and $D$ are Hessian matrices of two real-valued functions.
 Note that
\begin{equation}
\label{Siik}
S^{ii}_k(A) = S_{k-1}(A[\{i\}^c]).
\end{equation}
 \begin{proof}[Proof of Theorem \ref{iner_thm} using G\r{a}rding's inequality] If $A\in M_n(\R)$ is $k$-positive then $A$ is $m$-positive for all $1\leq m\leq k$. 
Thus, by an induction argument, it suffices to prove that if $A\in M_n(\R)$ is $k$-positive then $S_{k-1}(A[\{i\}^c])>0$  for all $i\in [n]$. 

Indeed,
if $\delta_i>0$, then $D = \diag(\delta_1,\cdots,\delta_n)$ is $k$-positive, and  we deduce from (\ref{Gij})
\begin{equation}
\label{Gij2}
\sum_{i=1}^n \delta_i S_{k-1} (A[\{i\}^c]))\geq  k [S_k(A)]^{\frac{k-1}{k}} [S_k(D)]^{\frac{1}{k}}. 
\end{equation}
For a fixed $i\in [n]$, letting $\delta_i=1$  and $\delta_j\rightarrow 0$ for $j\neq i$ in (\ref{Gij2}), we discover
$$S_{k-1}(A[\{i\}^c])\geq 0.$$
It remains to prove that $S_{k-1}(A[\{i\}^c]))\neq 0$ for all $i\in [n]$. Assume  that $S_{k-1}(A[\{1\}^c]))=0$. In this case, consider $\delta_1=1,$ $\delta_i=\e>0$ for $i=2,\cdots, k$ and $\delta_i=0$, otherwise. Then $D$ is still $k$-positive since $S_m(D)\geq \e^{m-1}>0$ for all $1\leq m\leq k$. Now,
(\ref{Gij2}) and  the assumption $S_{k-1}(A[\{1\}^c]))=0$ give
\begin{equation}
\label{iiep}
\e\sum_{i=2}^{k} S_{k-1}(A[\{i\}^c]) \geq  k [S_k(A)]^{\frac{k-1}{k}} \e^{\frac{k-1}{k}}. 
\end{equation}
Since $S_k(A)>0$,  by dividing both sides of (\ref{iiep}) by $\e$ and letting $\e\rightarrow 0^{+}$, we obtain
$$\sum_{i=2}^{k} S_{k-1}(A[\{i\}^c])  =\infty,$$
a contradiction.
\end{proof}
\begin{proof}[Proof  of Corollary \ref{cor2}]
Suppose $A, B\in M_n(\R)$ are $k$-positive. By Corollary \ref{cor1}, $D:=\diag(B)$ is $k$-positive. 
Applying (\ref{Gij}) to $A$ and $D= \diag(B)$, and recalling (\ref{Siik}), we find
$$\sum_{i=1}^n b_{ii} S_{k-1} (A[\{i\}^c]))\geq k [S_k(A)]^{\frac{k-1}{k}}  [S_k(\diag(B))]^{\frac{1}{k}}\geq  k [S_k(A)]^{\frac{k-1}{k}}  [S_k(B)]^{\frac{1}{k}} $$ 
where we used Theorem \ref{Hnk} in the last inequality.
\end{proof}
\section{Hyperbolic polynomials and a conjectural inequality}
 \label{HypSec}
 In this section, we state a generalization of Theorem \ref{Hnk} for hyperbolic polynomials. Using the theory of hyperbolic polynomials, we prove the convexity of $\Gamma_k(n)$ and the convexity of the set of $k$-positive matrices.
 \vglue 0.2cm
 First, we recall the concept of hyperbolic polynomials \cite{G} (see also \cite{HL1} for a self-contained account of G\r{a}rding's theory).
 \vglue 0.2cm
 Let $P$ be a homogeneous real polynomial of degree $k$ on $\R^n$. Given $a\in \R^n$, we say that $P$ is {\it $a$-hyperbolic} if $P(a)>0$, and for each $x\in\R^n,$ $P(ta + x)$ can be factored as
 $$P(ta + x) =P(a) \prod_{i=1}^k (t + \lambda_i(P; a,x))\quad \text{for all } t\in \R$$
 where $\lambda_i(P; a, x)$'s ($i=1, \cdots, k$) are real numbers, called {\it $a$-eigenvalues of $x$}.
  \vglue 0.2cm
  We recall the following fundamental theorem of hyperbolic polynomials; see \cite[Theorem 2]{G}.
  \begin{thm} [G\r{a}rding]
  Let $P$ be a homogeneous real polynomial of degree $k$ on $\R^n$. Assume that $P$ is $a$-hyperbolic. Denote the G\r{a}rding cone of $P$ at $a$ to be the set \[\Gamma_a(P)=\{x\in \R^n: \lambda_i(P; a, x)>0\text{ for all } i=1,\cdots, k\}.\]
  Then the following hold:
  \begin{enumerate}
  \item If $b\in \Gamma_a(P)$, then $P$ is $b$-hyperbolic and $\Gamma_a(P)=\Gamma_b(P)$.
  \item $\Gamma_a(P)$ is convex. 
  \end{enumerate}
 \end{thm}
 A self-contained proof of this theorem of G\r{a}rding can also be found in \cite{HL1} which consists of Theorems 3.6 and 5.1 there.
 
 Suppose now $P$ is $a$-hyperbolic. By G\r{a}rding's theorem, we can define the G\r{a}rding cone of $P$ to be
 $$\Gamma(P)=\{x\in \R^n: \lambda_i(P; a, x)>0\text{ for all } i=1,\cdots, k\},$$
 and $\Gamma(P)$ is independent of $a$.

 \begin{exam}[G\r{a}rding cone and $\Gamma_k(n)$]
 \label{Ex}
The $k$-th elementary symmetric function $S_k(\lambda)$ is a homogeneous real polynomial of degree $k$ on $\R^n$ and it is $\lambda$-hyperbolic at any $\lambda\in \Gamma_k(n)$. Moreover,
$$\Gamma(S_k)=\Gamma_k(n).$$
From the convexity of $\Gamma(S_k)$ due to G\r{a}rding's theorem, we deduce the convexity of $\Gamma_k(n)$
 from the above equality.
  \end{exam}
 \begin{proof}[Proof of the statements in Example \ref{Ex}] By Example 2, p. 959 in \cite {G}, we know that $S_k$ is $a$-hyperbolic where $a= (1, \cdots, 1)\in\R^n$. Thus, for any $x\in \R^n$, we have from the definition of $a$-hyperbolicity that the $a$-eigenvalues $\lambda_i (S_k; a, x)$ are real numbers, for all $i=1,\cdots, k$.

Assume  $x\in\Gamma_k(n)$. Then, from $\lambda_i (S_k; a, x)\in\R$, 
$$S_k(ta + x)=\sum_{i=0}^k {n-i\choose k-i}t^{k-i}S_i(x)= S_k(a)\prod_{i=1}^k (t+ \lambda_i (S_k; a, x))$$
and $S_i(x)>0$ for all $i=0, 1, \cdots, k$, we easily find that $\lambda_i (S_k; a, x)>0$ for all $i=1,\cdots, k$. Hence $x\in \Gamma(S_k)$ from which we deduce that $\Gamma_k(n)\subset\Gamma(S_k)$, and
$S_k$ is $x$-hyperbolic by G\r{a}rding's theorem. 
Recall that we use $\Gamma(S_k)$ to denote the G\r{a}rding cone of $S_k$ at $a=(1,\cdots, 1)$. 

Now, assume $x\in \Gamma(S_k)$. Then, by the definition of $\Gamma(S_k)=\Gamma_a(S_k)$, we have $\lambda_i (S_k; a, x)>0$ for all $i=1,\cdots, k$. Therefore, from the above expansion of $S_k(ta + x)$, we obtain $S_i(x)>0$ for all $i=1,\cdots, k$ which shows that $x\in \Gamma_k(n)$, or $\Gamma(S_k)\subset \Gamma_k(n)$.

Thus, we have $\Gamma(S_k)=\Gamma_k(n)$.
\end{proof}
A different proof of the convexity of $\Gamma_k(n)$ can be found in Section 2 of \cite{U}.
\vglue 0.2cm
Example \ref{Ex} shows that $k$-positive matrices are those having eigenvalues lying in the G\r{a}rding cone of $S_k$. 
 \vglue 0.2cm 
 \begin{exam}[G\r{a}rding cone and the set of $k$-positive matrices]
 \label{Pkex}
 Let $N= \frac{1}{2}n(n+1)$ and let $A= (a_{ij})_{1\leq i, j\leq n} \in M_n(\R)$ be symmetric. We can view $A$ as a point in $\R^N$. Then $P(A)=\det A$ is $A$-hyperbolic for any positive definite matrix $A$. Let $I_n$ be the identity $n\times n$ matrix. Define $P_k$ by
\[
\det (tI_n + A)=P(t I_n + A)= \sum_{k=0}^n t^{n-k} P_k(A) \quad\text{for all } t\in\R.
\]
Then  $P_k$ is a homogeneous polynomial of degree $k$ on $\R^N$; moreover, $P_k$ is $I_n$-hyperbolic. This follows from Example 3 and the discussion at the end of p. 959 in \cite{G}.

Note that \[P_k(A)= S_k(\lambda(A)).\] Arguing as in the proof of statements in Example \ref{Ex}, we have
\[\Gamma(P_k)=\{A\in M_n(\R):\lambda (A)\in\Gamma_k(n)\};\]
See also equation (2.10) in \cite{L}.
From the convexity of $\Gamma(P_k)$ due to G\r{a}rding's theorem, we deduce from the above equality the convexity of the set of $k$-positive matrices.
 
 \end{exam}
  \vglue 0.2cm 
 \begin{exam}
 \label{pexam}
 In many geometric problems (see, for example \cite{HL2, Sh, TW}), the Hessian equation operators $S_k$ are replaced by other hyperbolic polynomials $P$. One example is
 $$\mathcal{P}_p(\lambda)=\prod_{1\leq i_1<\cdots<i_p\leq n} (\lambda_{i_1}+\cdots+\lambda_{i_p}),\text{for }\lambda=(\lambda_1,\cdots,\lambda_n)\in\R^n.$$
 Note that $\mathcal{P}_1= S_n$ while $\mathcal{P}_n= S_1$. Moreover, $\mathcal{P}_{n-1}(\lambda(A))=\det (S_1(A) I_n-A)$.
 \end{exam}
 We note that the statement of Theorem \ref{pcor}, without any appeal to hyperbolic polynomials, is modeled on the hyperbolic polynomial $\mathcal{P}_p$ in Example \ref{pexam}.

 It is of interest to study matrices whose eigenvalues lying in the G\r{a}rding cone of a hyperbolic polynomial $P$ other than $S_k$ and $\mathcal{P}_p$. In this regard, we state the following generalization of Theorem \ref{Hnk}.
\begin{conj}[Hadamard-type inequalities for hyperbolic polynomials] 
\label{Conj1}
 Let $P$ be a homogeneous, real, symmetric, hyperbolic polynomial of degree $k$ on $\R^n$. Let $A\in M_n(\R)$. If $\lambda(A)\in \Gamma(P)$ then $(a_{11},\cdots, a_{nn})\in\Gamma(P)$ and
$$P(a_{11},\cdots, a_{nn})\geq P(\lambda(A)).$$
\end{conj}

{\bf Acknowledgements.} The author warmly thanks the referee for providing constructive comments that help improve the exposition of this note. The author is grateful to Trieu Le for useful suggestions.

\end{document}